\documentclass[10pt,twoside,a4paper,reqno]{amsart}
\usepackage{amsmath,amssymb,mathrsfs}

\usepackage{sseq}

\usepackage{mathrsfs}
\usepackage[all]{xy}

\newdir{ >}{{}*!/-10pt/\dir{>}}

\def\smono[#1]{\ar@{->}[#1]|@{|}}
\newcommand{\xymono}{\ar@{ >->}}
\newcommand{\xyepi}{\ar@{->>}}

\let\oldtocsection=\tocsection

\let\oldtocsubsection=\tocsubsection

\let\oldtocsubsubsection=\tocsubsubsection

\renewcommand{\tocsection}[2]{\hspace{0em}\oldtocsection{#1}{#2}}
\renewcommand{\tocsubsection}[2]{\hspace{1em}\oldtocsubsection{#1}{#2}}
\renewcommand{\tocsubsubsection}[2]{\hspace{2em}\oldtocsubsubsection{#1}{#2}}

\newtheorem{theorem}{Theorem}[section]
\newtheorem{lemma}[theorem]{Lemma}
\newtheorem{corollary}[theorem]{Corollary}
\newtheorem{proposition}[theorem]{Proposition}

\theoremstyle{definition}

\newtheorem{remark}[theorem]{Remark}

\numberwithin{equation}{section}

\newcommand{\spec}{\mathrm{Spec\hskip .5mm }}

\newcommand{\Z}{\mathbb{Z}}

\newcommand{\HL}{\mathcal{HL}}

\newcommand{\colim}{\operatorname{colim}}
\newcommand{\Spec}{\operatorname{Spec}}

\newcommand{\Hom}{\operatorname{Hom}}

\newcommand{\Fun}{\operatorname{Fun}}
\newcommand{\Vect}{\operatorname{Vect}}

\newcommand{\Ch}{\operatorname{Ch}}

\newcommand{\aaa}{\mathbb{A}}
\newcommand{\ppp}{\mathbb{P}}

\newcommand{\rk}{\operatorname{rk}}

\newcommand{\GW}{\operatorname{\mathcal{GW}}}
\newcommand{\KSp}{\operatorname{\mathcal{K}Sp}}

\newcommand{\F}{\mathbb{F}}

\newcommand{\cone}{\operatorname{cone}}

\newcommand{\proj}{\operatorname{proj}}

\newcommand{\on}{\hspace{.5ex}\operatorname{on}\hspace{.5ex}}

\newcommand{\Met}{\operatorname{Met}}
\newcommand{\qis}{\operatorname{qis}}
\newcommand{\ooo}{\mathfrak{o}}

\usepackage{amsmath,pdftexcmds}

\DeclareFontFamily{U}{cbgreek}{}
\DeclareFontShape{U}{cbgreek}{m}{n}{
        <-6>    grmn0500
        <6-7>   grmn0600
        <7-8>   grmn0700
        <8-9>   grmn0800
        <9-10>  grmn0900
        <10-12> grmn1000
        <12-17> grmn1200
        <17->   grmn1728
      }{}
\DeclareFontShape{U}{cbgreek}{bx}{n}{
        <-6>    grxn0500
        <6-7>   grxn0600
        <7-8>   grxn0700
        <8-9>   grxn0800
        <9-10>  grxn0900
        <10-12> grxn1000
        <12-17> grxn1200
        <17->   grxn1728
      }{}

\DeclareRobustCommand{\Qoppa}{%
  \text{\usefont{U}{cbgreek}{\normalorbold}{n}\symbol{21}}%
}
\makeatletter
\newcommand{\normalorbold}{%
  \ifnum\pdf@strcmp{\math@version}{bold}=\z@ bx\else m\fi
}
\makeatother

\title[$\Qoppa^{gs}$ vs $\Qoppa^s$]{Symmetric versus genuine symmetric forms in Hermitian $K$-theory}
 \author{Marco Schlichting}
 \address{Marco Schlichting, Mathematics Institute,
Zeeman Building,
University of Warwick,
Coventry CV4 7AL, UK} 

\thanks{}

\email{m.schlichting@warwick.ac.uk}



\begin{document}
\bibliographystyle{alpha}

\begin{abstract}
We show that for finite dimensional regular Noetherian rings that contain a field or are smooth over a Dedekind domain, the comparison map from the Hermitian $K$-theory of genuine symmetric forms to that of  symmetric forms is an isomorphism in degrees $\geq -1$ and a monomorphism in degree $-2$.
In particular, the spaces of Hermitian $K$-theory of genuine symmetric forms and the symplectic $K$-theory space are homotopy invariant for such rings.
\end{abstract}

\maketitle

\tableofcontents

\section{Introduction}

Hermitian $K$-theory is the study of the homotopy type of spaces associated with symplectic and symmetric forms and their automorphisms.
For instance, for a local commutative ring $R$, the genuine symplectic and the genuine symmetric Hermitian $K$-theory spaces $\KSp(R)$ 
and $\GW(R)$ have the homotopy types of  
$$\Z \times BSp(R)^+\hspace{4ex}\text{and} \hspace{4ex}GW_0(R) \times BO_{\infty,\infty}(R)^+$$ 
(the latter if the units of $R$ can be represented by differences of sums of squares, e.g, $2\in R^*$, $R=\Z_{(2)}$, or $R$ a perfect field of characteristic $2$).
Here, the groups $O_{\infty,\infty}(R)=\colim_nO_{n,n}(R)$ and $Sp(R)=\colim_nSp_{2n}(R)$ are the unions of the orthogonal  $O_{n,n}(R)$ and symplectic groups $Sp_{2n}(R)$ that are the automorphism groups of the bilinear forms $\left(\begin{smallmatrix}1 & 0 \\ 0 & -1\end{smallmatrix}\right)^{\perp n}$ and $\left(\begin{smallmatrix} 0 & 1 \\ -1 & 0 \end{smallmatrix}\right)^{\perp n}$ on $R^{2n}$, and $+$ denotes Quillen's plus construction \cite{Grayson:KII}.
Moreover, $GW_0(R)$ is the Grothendieck-Witt group of $R$, that is, the abelian group associated with the monoid of isometry classes of non-degenerate symmetric forms on finitely generated projective $R$-modules.
For a general commutative ring $R$, the spaces $\GW(R)$ and $\KSp(R)$ are the group completions, aka $S^{-1}S$-constructions \cite{Grayson:KII}, of the symmetric monoidal groupoids $S$ of non-degenerate symmetric bilinear and symplectic forms, respectively.

When $2\in R^*$, the spaces $\GW(R)$ and $\KSp(R)$ are reasonably well understood \cite{karoubi:annals}, \cite{myJPAA} and form part of a cohomology theory that satisfies Nisnevich descent, projective bundle formulas, blow-up formulas and is homotopy invariant on regular rings (with $2\in R^*$); see \cite{myJPAA}.
In particular, they are represented in Morel-Voevodsky's unstable $\aaa^1$-homotopy category \cite{Hornbostel}, \cite{GirjaMe} which is the starting point for its use in the classification of vector bundles on smooth affine algebras \cite{AsokFaselICM}.

Dispensing with the hypothesis $2\in R^*$, due to the recent work of \cite{CDH+I}, \cite{CDH+II}, \cite{CDH+III}, \cite{HebestreitSteimle}, the spaces  $\GW(R)$ and $\KSp(R)$ are now the infinite loop spaces of spectra $GW^{[0]}(R)$ and $GW^{[2]}(R)$ that form part of a sequence $GW^{[r]}(R)$, $r\in \Z$, that are related to algebraic $K$-theory $K(R)$ 
by the {\em fundamental fibre sequence} of spectra \cite[Main Theorem]{CDH+II}
\begin{equation}
\label{eq:FundSeq}
GW^{[r]}(R) \stackrel{\operatorname{hyp}}{\longrightarrow} K(R) \stackrel{\operatorname{fgt}}{\longrightarrow} GW^{[r+1]}(R)
\end{equation}
generalising Bott periodicity as formulated in \cite{myJPAA}.
Here, we donote by $GW^{[r]}(R)$ the Hermitian $K$-theory spectrum 
$$GW^{[r]}(R)=GW(R,(\Qoppa^{gs})^{[r]})$$ 
of $r$-th shifted {\em genuine symmetric forms} on $R$ as defined in \cite{CDH+I}, \cite{CDH+II}, \cite{CDH+III}.
The spectra $GW^{[r]}(R)$ have been generalised to quasi-compact and quasi-separated schemes $X$ in \cite{CHN} to yield spectra $GW^{[r]}(X) = GW(X,(\Qoppa^{gs})^{[r]})$
whose Karoubi-localizing variants satisfy Nisnevich descent and a projective line bundle formula. 
When $X$ is divisorial and $r=0$, the infinite loop space $\GW^{[0]}(X)$ of $GW^{[0]}(X)$ is the Hermitian $K$-theory space studied in \cite{myMV} which is the Hermitian $K$-theory in the sense of \cite{myHermKex} of the exact category of finite rank vector bundles on $X$ with duality $E \mapsto E^*=Hom(E,O_X)$; see \cite{CHN}, \cite{FormIII}.

In the series of papers \cite{CDH+I}, \cite{CDH+II}, \cite{CDH+III}, \cite{CHN} the authors also construct a variant of Hermitian K-theory which we will denote by $KO^{[r]}(X)$. 
It is the Hermitian $K$-theory spectrum
$$KO^{[r]}(X)=GW(X,(\Qoppa^s)^{[r]})$$ 
of $r$-th shifted {\em symmetric} forms on $X$; see \cite{CHN}.
The theories $KO^{[r]}$ seem to have better properties than $GW^{[r]}$.
Indeed, besides the fundamental fibre sequence \cite[Main Theorem]{CDH+II}, Nisnevich descent (for its Karoubi-localizing variant) and projective bundle formulas, the authors also show in \cite{CHN} that $KO^{[r]}$ satisfies devissage and is homotopy invariant on finite dimensional regular Noetherian schemes.
In particular, these theories are representable in Morel-Voevodsky's $\aaa^1$-homotopy theory without assuming $2 \in \Gamma(X,O_X)^*$.

Whereas $GW^{[r]}(X)$ is linked to the study of classical objects of symmetric and symplectic forms (for $r=0,2$), a geometric meaning of $KO^{[r]}(X)$ is less clear.
However, the canonical natural transformation of quadratic functors $\Qoppa^{gs} \to \Qoppa^{s}$ from genuine symmetric to symmetric forms \cite{CDH+I} induces a natural comparison map 
$$GW^{[r]}(X) \to KO^{[r]}(X)$$
of associated Hermitian $K$-theories.
When $2 \in \Gamma(X,O_X)^*$, this map is an equivalence \cite{CDH+I}, essentially by definition.
For general rings (with involution), the comparison map was studied in detail in \cite{CDH+III} where it was shown that for regular Noetherian rings of dimension $d$
the comparison map is an isomorphism in degrees $\geq d+r-1$ and a monomorphism in degree $d+r-2$.
Our main results, Theorems \ref{thm:MainL} and \ref{thm:MainGW} below, improve these bounds.
In particular, our bounds do not depend on the dimension $d$ of $R$.
The caveat is that we only consider rings with trivial involution (or with involutions twisted by a line bundle).
We denote by $L_n^{gs}(R)$ and $L_n^s(R)$ the $L$-theory groups of genuine symmetric and symmetric forms on $R$; see \cite{CDH+III}.

\begin{theorem}
\label{thm:MainL}
Let $R$ be a regular Noetherian domain of finite Krull dimension. 
If $2\notin R^*$ assume that $R$ contains a field or is smooth over a Dedekind domain.
Then the natural map
\begin{equation}
\label{eq:thm:MainL}
L_n^{gs}(R) \to L_n^s(R)
\end{equation}
is an isomorphism for $n \geq -1$ and a monomorphism for $n=-2$.
If moreover, $2\in R$ is a non-zero divisor, then the map (\ref{eq:thm:MainL}) is an isomorphism for $n \geq -2$ and a monomorphism for $n=-3$.
\end{theorem}

We prove Theorem \ref{thm:MainL} in a slightly more general situation below; see Theorem \ref{thm:MainLText}.
All we need to assume is that the mixed characteristic $(0,2)$ stalks $R_{\wp}$ of $R$ are weakly regular over a dvr, that is, they contain a dvr with unifomizer $\pi$ which is part of a regular system of parameters for $R_{\wp}$.

We note that the statement of the theorem is false, in general, if the involution on $R$ is non-trivial; see \cite[Example 1.3.11]{CDH+III}.
Even if the involution on $R$ is trivial, the comparison map is not an isomorphism for $R=\Z/4$ and all $n\geq -2$; see \cite{Read}.
The bounds in Theorem \ref{thm:MainL} are sharp in characteristic $2$.
Indeed, for a local ring $R$ containing a field of characteristic $2$, we have $L_{-2}^{gs}(R)=0$ (by Lemma \ref{lem:Lgs-1} below) whereas $L_{-2}^{s}(R)=L_0^{s}(R)\neq 0$, by definition.

As a corollary of Theorem \ref{thm:MainL}, we obtain the following; see Corollary \ref{cor:MainGWText} in the text.

\begin{theorem}
\label{thm:MainGW}
Let $X$ be a finite dimensional integral regular Noetherian separated scheme.
If $2 \notin \Gamma(X,O_X)^*$ assume that $X$ is defined over a field or is smooth over a Dedekind domain.
Then the comparison maps from genuine symmetric to symmetric Hermitian $K$-theory
\begin{equation}
\label{eq:thm:MainGW}
GW^{[r]}_n(X) \to KO^{[r]}_n(X)
\end{equation}
are isomorphisms for $n \geq r-1$ and a monomorphism for $n=r-2$.
If, moreover, $0\neq 2 \in \Gamma(X,O_X)$, then the maps (\ref{eq:thm:MainGW}) are isomorphisms for $n \geq r-2$ and a monomorphism for $n=r-3$.
\end{theorem}

Again, we only really need to assume that the mixed characteristic $(0,2)$ stalks of $X$ are weakly regular over a dvr; see Corollary \ref{cor:MainGWText}.
In particular, for such schemes, the theories of genuine symmetric forms $GW^{[0]}$ and of symmetric forms $KO^{[0]}$ only differ in negative degrees, and the spectra $GW^{[0]}$ and $KO^{[0]}$ are simply different deloopings of the same hermitian $K$-theory space $\GW$ of genuine symmetric forms.

\begin{corollary}
\label{cor:htpyInv}
The hermitian $K$-theory space of genuine symmetric forms $\GW$ and the symplectic $K$-theory space $\KSp$ are homotopy invariant on regular Noetherian rings and separated schemes of finite Krull dimension whose mixed characteristic $(0,2)$ stalks are weakly regular over a dvr.
That is,
for such a regular Noetherian ring $R$, the inclusion $R \subset R[T]$ induces homotopy equivalences
$$\GW(R) \stackrel{\sim}{\longrightarrow} \GW(R[T]),\hspace{6ex} \KSp(R) \stackrel{\sim}{\longrightarrow} \KSp(R[T]).$$
\end{corollary}

In consequence, the classical Witt group of symmetric bilinear forms $W(X)$ of Knebusch \cite{Knebusch} and the integral homology $H_*(Sp(R))$ of the infinite symplectic group are homotopy invariant on regular Noetherian schemes and rings $X$ and $R$ of finite dimension that are smooth over a Dedekind domain or that are defined over a field.
In particular, the genuine symmetric and symplectic $K$-theory spaces are respresentable in Morel-Voevodsky's unstable $\aaa^1$-homotopy theory over a field or Dedekind domain.
See Remark \ref{rmk:motivicRep} for a discussion of genuine Hermitian $K$-theory in the stable motivic homotopy category.
\vspace{2ex}

Our proofs of Theorems \ref{thm:MainL} and \ref{thm:MainGW} are inspired by the literature around the Gersten conjecture \cite{CHK}, \cite{SchmidtStrunk}.
But our approach is simpler.
Instead of Nisnevich descent we will use Flat excision (Proposition \ref{prop:FlatExc}), and we replace Gabber's Geometric Presentation Lemma with the use of the Cohen Structure Theorems for complete regular local rings (Proposition \ref{prop:CohenFlattening}).
Also, we need to work with the homotopy fibre of $L^{gs} \to L^s$ since the Gersten conjecture fails to hold for $L^{gs}$ and $L^s$ in the mixed characteristic $(0,2)$ case; see Remark 
\ref{rem:LThGersten}.

\section{Some low degree computations}

\vspace{2ex}
{\em All rings in this article are assumed to be commutative.}

\begin{lemma}
\label{lem:Lgs-1}
Let $R$ be a local ring. Then $$GW^{[2]}_1(R)=GW^{[1]}_0(R)=L^{gs}_{-1}(R)=L^{gs}_{-2}(R)=0$$
\end{lemma}

\begin{proof}
Recall that for local rings, the symplectic $K_0$ is $\Z$ as every non-degenerate symplectic form over a local ring is hyperbolic, and its symplectic $K_1$ vanishes.
Now, the lemma follows from the exact sequence of the fundamental sequence (\ref{eq:FundSeq})
$$0=GW_1^{[2]}(R) \longrightarrow GW^{[1]}_0(R)  \stackrel{0}{\longrightarrow} K_0(R)  \stackrel{\cong}{\longrightarrow} GW_0^{[2]} (R) \longrightarrow L^{gs}_{-2}(R)  \to 0$$
and the surjection $GW^{[n]}_0(R)  \twoheadrightarrow  L^{gs}_{-n}(R) $. 
\end{proof}

\begin{lemma}
\label{lem:L-1DVR}
Let $V$ be a DVR or a field, then $L^s_{-1}(V)=L^q_{-3}(V)=L^{gs}_{-3}(V)=0$.
If $V$ is a DVR of mixed characteristic $(0,2)$, then $L^s_{-2}(V)=0$.
\end{lemma}

\begin{proof}
The $L^q$ statements can be deduced from \cite[Proposition 2.3.9]{CDH+III} and its proof.
For fields, we have $L^s_{-1}=0$, by \cite[Corollary 1.3.8]{CDH+III}.
Moreover, $L^q_{-3}=L^{gs}_{-3}$ for any ring \cite[Corollary 1.2.12]{CDH+III}.
Let $V$ now be a dvr, and let $F$ and $k$ be the field of fractions and the residue field of $V$.
Consider the exact localisation sequence (using devissage for $L^s$)
$$L^s_0(F) \twoheadrightarrow L^s_0(k) \stackrel{0}{\longrightarrow} L_{-1}^s(V) \to L_{-1}^s(F)=0$$
where $L^s_{-1}(F)=0$, by \cite[Corollary 1.3.8]{CDH+III}, and $L^s_0(F) \twoheadrightarrow L^s_0(k)$ is surjective since it is $L^{gs}_0(F) \twoheadrightarrow L^{gs}_0(k)$ which is well know to be surjective, for instance since the next term in the exact sequence is $L^{gs}_{-1}(V)$ which vanishes, by Lemma \ref{lem:Lgs-1}.
Finally, let $V$ be a DVR of mixed characteristic $(0,2)$ with field of fractions $F$ of characteristic $0$.
Then  the exact localisation sequence (using devissage for $L^s$)
$$0= L^s_{-1}(k) \to L_{-2}^s(V) \to L_{-2}^s(F)=0$$
shows that $L_{-2}^s(V)=0$.
\end{proof}

\begin{corollary}
\label{cor:LvsLDVR}
Let $R$ be a field or a DVR, then the canonical map $L^{gs}_n(R) \to L^{s}_n(R)$ is an isomorphism for $n\geq -1$ and a monomorphism $n=-2$.
If moreover $0\neq 2\in R$, then $L^{gs}_n(R) \to L^{s}_n(R)$ is an isomorphism for $n=-2$ and a monomorphism for $n=-3$.
\end{corollary}

\begin{proof}
Assume first that $n\geq -2$.
For fields, the statement is \cite[Theorem 6, Corollary 1.3.10]{CDH+III}, and for DVRs the statement holds for $n\geq 0$, by loc.cit.
The other statements follow from Lemmas \ref{lem:Lgs-1} and \ref{lem:L-1DVR}.
\end{proof}

\begin{lemma}
\label{lem:GW10ZX}
We have $GW^{[1]}_0(\Z[X])=GW^{[1]}_0(\Z)=0$.
\end{lemma}

\begin{proof}
By \cite[Corollary 1.4]{Grunewald+MathZ}, we have $GW^{[2]}_1(\Z[X]) = K_1Sp(\Z[X]) = K_1Sp(\Z)=0$.
We conclude using the fundamental exact sequence (\ref{eq:FundSeq}) in low degrees
$$GW^{[2]}_1(\Z[X]) \to GW^{[1]}_0(\Z[X]) \to K_0(\Z[X]) \rightarrowtail GW^{[2]}_0(\Z[X])$$
where the last map is injective since composition with the forgetful map $GW^{[2]}_0(\Z[X]) \to K_0(\Z[X])=\Z$ is multiplication by $2$ which is injective. 
As a direct summand of $GW^{[1]}_0(\Z[X])$, the group $GW^{[1]}_0(\Z)$ vanishes too.
\end{proof}

\section{Flat excision and the Projective line bundle formula}

Tensor product of genuine symmetric forms makes the spectrum $GW^{[0]}(X)$ into a commutative ring spectrum.
More generally, there are natural parings of spectra
$$\cup: GW^{[r]}(X) \wedge GW^{[s]}(X) \to GW^{[r+s]}(X)$$
such that $x \cup y = (-1)^{pq}\langle -1\rangle^{rs}y \cup x$ for $x \in GW^{[r]}_p(X)$ and $y\in GW^{[s]}_q(X)$.
Consider the map of exact sequences of abelian groups
$$\xymatrix{
K_0(\Z)\hspace{1ex} \ar@{>->}[r] \ar[d]^1 & GW^{[-4]}_0(\Z) \ar@{->>}[d]^{\rk} \ar@{->>}[r] & L^{gs}_4(\Z) \ar@{->>}[d]^{\rk} \ar[r] & 0\\
K_0(\Z)\hspace{1ex}  \ar@{>->}[r]_2 & K_0(\Z) \ar@{->>}[r] & \Z/2 \ar[r] & 0}
$$
where the top row is the exact sequence defining $L_4^{gs}$, and the vertical maps are the forgetful rank maps.
The right square is cartesian, and there is a unique element $\sigma \in GW^{[-4]}_0(\Z)$ that is sent to $1 \in K_0(\Z)=\Z$ and $1 \in L^{gs}_4(\Z)=L_4^s(\Z)=W(\Z)=\Z$.
The element $\sigma$ considered as an element of $KO^{[-4]}_0(\Z)$ is invertible in the bigraded theory $KO^{[*]}_*(\Z)$ since it is invertible in $K(\Z)$ and $L^{s}(\Z)$.
It follows that the natural map $GW \to KO$ factors through the map 
$$GW(X)[\sigma^{-1}] \stackrel{\sim}{\longrightarrow} KO(X)$$
of bigraded theories which is an equivalence since $L^{gs}[\sigma^{-1}]\cong L^s$.
Similarly, let $\eta\in GW^{[-1]}_{-1}(\Z)$ be the image of $\langle 1 \rangle \in GW^{[0]}_0(\Z)$ under the boundary map $GW^{[0]}_0(\Z) \to GW^{[-1]}_{-1}(\Z)$ of the fundamental sequence (\ref{eq:FundSeq}) for $\Z$.
Then 
$$L^{gs}(X) \simeq GW(X)[\eta^{-1}],\hspace{4ex} L^s(X) =KO(X)[\eta^{-1}] = GW(X)[(\sigma\eta)^{-1}].$$
In particular, $KO$, $L^{gs}$ and $L^s$ are algebra spectra over $GW$.
For a closed subset $Z \subset X$ of a scheme $X$ and a functor $F$ from (divisorial) schemes to spectra, we denote by $F(X\on Z)$ the fibre of the restriction map $F(X) \to F(X-Z)$.
Furtheremore, we denote by $HL$ the fibre of $L^{gs} \to L^s$.
Thus, we have spectra $GW^{[n]}(X\on Z)$, $L^{gs}(X\on Z)$, $KO^{[n]}(X\on Z)$, $L^{s}(X\on Z)$ and $HL(X\on Z)$ all of which are module spectra over $GW(X)$.
\vspace{1ex}

Our proof of Theorems \ref{thm:MainL} and \ref{thm:MainGW} relies on the comparison of Grothendieck-Witt and $L$-theory of a local ring and its completion which is achieved by means of the following proposition.
In case the map $A \to B$ is etale, the theorem is also a consequence of Nisnevich descent proved in \cite[Corollary 4.4.2]{CHN}.

\begin{proposition}[Flat excision]
\label{prop:FlatExc}
Let $(A, m_A,k_A) \to (B,m_B,k_B)$ be a flat homomorphism of regular local rings such that $m_B=m_AB$.
Assume that the map on residue fields is an isomorphism $k_A \cong k_B$.
Then for all integers $r$, the map on relative Grothendieck-Witt spectra 
$$GW^{[r]}(A \on m_A) \stackrel{\sim}{\longrightarrow} GW^{[r]}(B\on m_B)$$
is an equivalence.
The same holds for $KO^{[r]}$, $L^s$, $L^{gs}$ and $HL$ in place of $GW^{[r]}$.
\end{proposition}

\begin{proof}
The proof works without the regularity assumption for any Noetherian local rings $A$ and $B$ (provided we replace $GW$ with its Karoubi localising version $\mathbb{GW}$).
We prove the statement first for $GW^{[r]}$.
Since the proposition holds for K-theory \cite[Proposition 3.19]{TT} in place of $GW^{[r]}$, it suffices to prove the case $r=0$ in view of the fundamental fibre sequence (\ref{eq:FundSeq}).
For a divisorial scheme $X$, denote by $D_b(X) = \qis^{-1}\Ch_b(X)$ the derived (stable infinity) category of bounded chain complexes of finite rank vector bundles over $X$.
If $(R,m_R)$ is a local ring, denote by $D_b(R\on m_R) = \qis^{-1}\Ch_b(R\on m_R)$ the derived (stable infinity) category of bounded chain complexes of finitely generated free $R$-modules with finite length homology $\Ch_b(R\on m_R)$.
For regular local rings $(R,m_R)$, the Poincare-Verdier sequence \cite[Corollary 4.2.12]{CHN}
\begin{equation}
\label{eqn:FlatExcVerdier}
(D_b(R\on m_R),\Qoppa_R^{gs})\to (D_b(R),\Qoppa_R^{gs}) \to (D_b(\Spec R - m_R),\Qoppa_{\Spec R - m_R}^{gs})
\end{equation}
shows that the fibre $GW^{[0]}(R \on m_R)$ of $GW^{[0]}(R) \to GW^{[0]}(\Spec R - m_R)$ is the hermitian $K$-theory of genuine symmetric forms of the category $(Ch_b(R\on m_R),\qis)$.
Here we used regularity of $R$ to ensure that the second map in (\ref{eqn:FlatExcVerdier}) is essentially surjective.

The main point of Thomason's proof of the $K$-theory analogue in \cite[Proposition 3.19]{TT} is that the map $A \to B$ induces an equivalence of associated derived categories
$D_b(A\on m_A) \stackrel{\sim}{\longrightarrow} D_b(B\on m_B)$.
For regular Noetherian local rings, this is also a consequence of the fact that in this case $D_b(R\on m_R)$ is the bounded derived category of the abelian category $M_{fl}(R)$ of finite length $R$-modules, and $B\otimes_A:M_{fl}(A) \to M_{fl}(B)$ is an equivalence of categories.
Thus, in order to show that $B\otimes_A: (D_b(A\on m_A),\Qoppa^{gs}_A) \to (D_b(B\on m_B),\Qoppa^{gs}_B)$ is an equivalence of Poincare infinity categories (and hence induces an equivalence of $GW$-spectra), we need to show that for any complex $C \in D_b(A\on m_A)$ the map of spectra of quadratic forms $\Qoppa^{gs}_A(C) \to \Qoppa^{gs}_B(B\otimes_AC)$ is an equivalence.
It suffices to show that $\Omega^{\infty}\, \Qoppa^{gs}_A(C) \to \Omega^{\infty}\, \Qoppa^{gs}_B(B\otimes_AC)$ is an equivalence for all $C \in D_b(A\on m_A)$ as the functor $\Qoppa^{gs}_A$ is determined by the functor $\Omega^{\infty}\, \Qoppa^{gs}_A$ as a canonical (non-connective) delooping.

Recall from \cite{FormIII} that for $C \in \Ch_b(R\proj)$, we have
$$\Omega^{\infty}\, \Qoppa^{gs}_R(C) = Q^{gs}(\Delta^{\bullet}C),\hspace{3ex}\text{
where}\hspace{3ex}Q^{gs}_R(C)=\Hom_R(C,C^{\sharp_R})^{C_2},$$
 $C^{\sharp_{R}}$ is the internal hom complex $C^{\sharp_{R}} =[C,R]_R$, and $\Delta^{\bullet}C$ is the cosimplicial chain complex $n \mapsto \Delta^{n}C$ which in degree $n$ is the tensor product over $\Z$ of $C$ with the normalized chain complex of the standard simplicial set $\Delta^n$.
Note that $\Hom_R(E_1,E_2)$ is naturally an $R$-module for any $E_i \in \Ch_b(R\proj)$.
In particular, this is the case for $E_1=C$ and $E_2=C^{\sharp_R}$, and  $Q^{gs}_R(C)=\Hom_R(C,C^{\sharp_R})^{C_2}$ is an $R$ module as the fixed set under a $C_2$ action through $R$-module homomorphisms.
This defines a linear action\footnote{The linear action is different from the action of $R$ on $Q^{gs}_R(C)$ as a quadratic functor.
The latter is given by the square of the linear action.}  of $R$ on $Q^{gs}_R(C)$.
In the following, we consider $Q^{gs}_R$ equipped with this linear action.
For $C \in \Ch_b(R\proj)$, we have
$
Q^{gs}_B(C\otimes_AB) = \Hom_B(C\otimes_AB,(C\otimes_AB)^{\sharp_B})^{C_2} = \Hom_A(C,[C,B]_A)^{C_2}  = (\Hom_A(C,C^{\sharp_A})\otimes_AB)^{C_2} = (\Hom_A(C,C^{\sharp_A})^{C_2}\otimes_AB
$
where the last equation holds since $B$, being flat over $A$, is a filtered colimit of finitely generated free $A$-modules.
The map $\Omega^{\infty}\, \Qoppa^{gs}_A(C) \to \Omega^{\infty}\, \Qoppa^{gs}_B(C\otimes_AB)$ therefore is the map
\begin{equation}
\label{eq:FlatExcAlt}
Q^{gs}_A(\Delta^{\bullet}C) \to Q^{gs}_A(\Delta^{\bullet}C)\otimes_AB: \xi \mapsto \xi \otimes 1
\end{equation}
which we have to show is a weak equivalence of simplicial sets in case $C$ has finite length homology.
We note that for such complexes $C$, the homotopy groups of $Q^{gs}_A(\Delta^{\bullet}C)$ are also finite length $A$-modules.
Indeed, they are finitely generated $A$-modules that are supported at the closed point $m_A$ of $\Spec A$ since
for $f\in m_A$, we have $Q^{gs}_A(\Delta^{\bullet}C)_f = Q^{gs}_{A_f}(\Delta^{\bullet}C_f)$ is contractible because $C_f \in D_b(A_f)$ is contractible and $Q^{gs}(\Delta^{\bullet})$ preserves chain homotopy equivalences \cite[Proposition 2.2.6]{FormIII}.
Finally, as mentioned above, for simplicial $A$-modules $E$ whose homotopy $A$-modules have finite length, the map $E \to E \otimes_AB:x \mapsto x\otimes 1$ is an equivalence.
Consequently, the map (\ref{eq:FlatExcAlt}) is an equivalence.
\end{proof}

In our proof of Theorems \ref{thm:MainL} and \ref{thm:MainGW}, we require a specific form of the projective line bundle formula for $GW$.
In \cite[Proposition 6.2.1]{CHN} (for $m=0$), the existence of a a split Poincare-Verdier sequence
\begin{equation}
\label{eq:P1CHN}
D_b(X,\Qoppa^{gs}) \stackrel{p^*}{\longrightarrow} D_b(\ppp^1_X,\Qoppa^{gs}) \stackrel{p_*(-\otimes O(-1))}{\longrightarrow} D_b(X,(\Qoppa^{gs})^{[-1]})
\end{equation}
was proved. 
We need to know that the splitting in (\ref{eq:P1CHN}) is given by the cup product with an element $\beta \in GW^{[1]}_0(\ppp^1_{\Z})$.
The element $\beta$ was defined in \cite[\S 9.4]{myJPAA} and exists over $\Z$.
Indeed, for a commutative ring $k$, consider the category $\Vect(\ppp^1)$ of finite rank vector bundles over the projective line $\ppp^1_R = \operatorname{Proj}(k[S,T])$ equipped with the duality $E\mapsto Hom_{O_{\ppp^1}}(E, O_{\ppp^1})$.
By functoriality, the category $\Met\Vect(\ppp^1)=\Fun([1],\Vect(\ppp^1))$ of morphisms in $\Vect(\ppp^1)$ is an exact category with duality.
In that category, we have an object $S:O_{\ppp^1}(-1) \to O_{\ppp^1}$ (multiplication by $S \in k[S,T]$) equipped with the symmetric form $\tilde{\beta}$
$$\xymatrix{\ar@{}[d]_{\tilde{\beta}:} &
O_{\ppp^1}(-1) \ar[r]^{S} \ar[d]_T & O_{\ppp^1}\ar[d]^T\\
& O_{\ppp^1} \ar[r]_S & O_{\ppp^1}(1)}$$
(multiplication by $T$).
We consider $\tilde{\beta}$ as a symmetric form in the exact category $\Met\Ch_b\Vect(\ppp^1) =\Fun([1],\Ch_b\Vect(\ppp^1))$ of morphisms of  complexes via the embedding of $\Vect(\ppp^1)$ into $\Ch_b\Vect(\ppp^1)$ as complexes concentrated in degree $0$.
Equipped with the point-wise quasi-isomorphisms, $(\Met\Ch_b\Vect(\ppp^1),\qis)$ is a model for the metabolic category \cite{CDH+II} of the stable infinity category $D_b\Vect(\ppp^1)$.
By \cite{CDH+II}, or by direct inspection using the canonical semi-orthononal decomposition of the middle term, we have a split Poincare-Verdier sequence
{\small
$$(\Ch_b\Vect(\ppp^1),\Qoppa^{gs},\qis) \stackrel{C \mapsto id_C}{\longrightarrow} (\Met\Ch_b\Vect(\ppp^1),\Qoppa^{gs},\qis) \stackrel{\cone}{\longrightarrow} (\Ch_b\Vect(\ppp^1),(\Qoppa^{gs})^{[1]},\qis).$$
}
In particular, if we denote by $w$ the set of maps in $\Met\Ch_b\Vect(\ppp^1)$ that become quasi-isomorphisms in $(\Ch_b\Vect(\ppp^1),\qis)$ upon applying the cone functor, we obtain an equivalence of Poincare infinity categories 
$$\cone: (\Met\Ch_b\Vect(\ppp^1),\Qoppa^{gs},w) \stackrel{\sim}{\longrightarrow} (\Ch_b\Vect(\ppp^1),(\Qoppa^{gs})^{[1]},\qis).$$
Since the map $\tilde{\beta}$ is a quasi-isomorphism in $\Ch_b\Vect(\ppp^1)$ it defines a non-degenerate form in $(\Met\Ch_b\Vect(\ppp^1),\Qoppa^{gs},w)$.
In particular, its cone $\beta_R=\cone(\tilde{\beta})$ defines a non-degererate form in $(\Ch_b\Vect(\ppp^1,(\Qoppa^{gs})^{[1]},\qis)$.
This defines the element 
$$\beta_R \in GW^{[1]}_0(\ppp_R^1).$$
For $R=\Z$, we simply write $\beta$ in place of $\beta_{\Z}$.
\vspace{1ex}

\begin{proposition}
\label{prop:P1bdl}
For any commutative ring $R$ and all integers $r\in \Z$, the following is an equivalence
\begin{equation}
\label{eq:P1}
(p^*,\beta\cup \underline{\phantom{x}}): GW^{[r]}(R) \oplus GW^{[r-1]}(R) \stackrel{\sim}{\longrightarrow} GW^{[r]}(\ppp^1_R)
\end{equation}
The same holds with $KO^{[r]}_n$ ($L^{s}_{n-r}$, $L^{gs}_{n-r}$, and $HL_{n-r}$, respectively) in place of $GW^{[r]}_n$.
\end{proposition}

\begin{proof}
Since the analogous theorem holds for $K$-theory in place of $GW^{[n]}$, it suffices to prove the proposition for $n=1$, in view of the fundamental sequence (\ref{eq:FundSeq}).
We need to check that the composition of duality preserving functors
$$(\Ch_b\Vect(R),\Qoppa^{gs},\qis) \stackrel{\otimes \beta}{\longrightarrow} (\Ch_b\Vect(\ppp^1),(\Qoppa^{gs})^{[1]},\qis) \longrightarrow (\Ch_b\Vect(\ppp^1),(\Qoppa_w^{gs})^{[1]},w)$$
is an equivalence of Poincare infinity categories where $w$ is the set of morphisms that are isomorphisms in the quotient category $D_b\Vect(\ppp^1)/p^*D_b\Vect(R)$
and $\Qoppa^{gs}_w$ is the Kan extension of $\Qoppa^{gs}$ from $D_b\Vect(\ppp^1)$ to the quotient $D_b\Vect(\ppp^1)/p^*D_b\Vect(R)$.
From the split Verdier sequence (\ref{eq:P1CHN}) of \cite{CHN}, there is an equivalence Poincare infinity categories
$$(\Ch_b\Vect(\ppp^1),(\Qoppa_w^{gs})^{[1]},w) \simeq (\Ch_b\Vect(R),\Qoppa^{gs},\qis)$$
and thus a functor of Poincare infinity categories 
$$f:(\Ch_b\Vect(R),\Qoppa^{gs},\qis) \stackrel{\otimes \beta}{\longrightarrow} (\Ch_b\Vect(\ppp^1),(\Qoppa_w^{gs})^{[1]},w) \simeq (\Ch_b\Vect(R),\Qoppa^{gs},\qis).$$
This functor is an equivalence of stable infinity categories.
Since the generator $R$ on the left side goes to the generator $R$ on the right side, the map $f$ of Poincare infinity categories is determined by the map on quadratic functors on the generator $R$
$$f_R: R = \Qoppa^{gs}(R) \to \Qoppa^{gs}(R) = R.$$
Since $f$ sends non-degenerate forms to non-degenerate forms, the map $f_R: \Qoppa^{gs}(R) \to \Qoppa^{gs}(R)$ sends units to units.
Moreover, that map is bilinear, hence an isomorphism, since $\Qoppa^{gs}(R) = B_{\Qoppa^{gs}}(R,R)=Hom_R(R,R)$ agrees with its bilinear part on $R$.
This proves the statement for $GW$.
The statements for $KO$, $L^{gs}$, $L^s$ and $HL$ follow by localising at $\sigma$, $\eta$, $\sigma\eta$ and taking homotopy fibres.
\end{proof}

\begin{proposition}
\label{prop:ZarDesent}
The fibre $HL$ of $L^{gs} \to L^s$ satisfies Zariski descent on divisorial schemes.
In particular, if $X$ is finite dimensional divisorial Noetherian scheme such that $HL_n(O_{X,x})=0$ for $n\geq n_0$ for all $x\in X$, then $HL_n(X)=0$ for all $n\geq n_0$.
\end{proposition}

\begin{proof}
The fibres of the map $L(\Qoppa) \to \mathbb{L}(\Qoppa)$ from $L$-theory to their Karoubi localising variants are independent of the quadratic functor $\Qoppa$ and equal the fibre of $K^{tC_2} \to \mathbb{K}^{tC_2}$; see \cite[\S 2.3]{CHN}.
In particular, $HL$ is also the fibre of the map $\mathbb{L}^{gs} \to \mathbb{L}^s$  of functors both of which satisfy Zariski descent \cite{CHN}.
For a scheme $X$ as in the proposition, we therefore have the strongly convergent Brown-Gersten spectral sequence \cite{BrownGersten}
$$H^p_{Zar}(X,\HL_{-q}(O_X)) \Rightarrow HL_{-p-q}(X)$$
where $\HL_{-q}(O_X)$ is the sheaf associated with the presheaf $U \mapsto HL_{-q}(\Gamma(U,O_X))$ in the Zariski topology.
Hence, the vanishing of $HL_n$ at the stalks of $X$ for $n\geq n_0$ implies $HL_n(X)=0$ for all $n\geq n_0$.
\end{proof}

\section{Agreement and Homotopy invariance}

Recall that $HL$ denotes the fibre of $L^{gs} \to L^s=L^{gs}[\sigma^{-1}]$.
The goal is to prove that $HL_n(R)=0$ for all regular rings $R$ as in Theorem \ref{thm:MainL} and $n\geq -2$ ($n\geq -3$ in the mixed characteristic case).
If we denote by $F$ the field of fractions of $R$, we would like to prove a version of the Gersten conjecture to show that $HL_n(R)$ injects in to $HL_n(F)$ and then use the known vanishing of the $HL_n$ groups for $F$.
We note though that we cannot directly work with the theories $L^{gs}$ and $L^s$ since the Gersten conjecture fails to hold for $L^{gs}$ and $L^s$ in the mixed characteristic $(0,2)$ case as the following example shows.

\begin{remark}
\label{rem:LThGersten}
Let $V$ be a dvr with characteristic $0$ field of fractions $F$ and characteristic $2$ residue field $k$.
In this case, the Gersten complexes for $L^{gs}$ and $L^s$ fail to be exact.
Indeed, localisation together with devissage \cite{CDH+III} yields the exact sequence
$$ L^s_2(F) \to L^s_2(k) \to L^s_1(V) \to L_1^s(F) \to L^s_1(k)$$
But $L^s_2(F)=L^s_1(F)=L^s_1(k)=0$ \cite{CDH+III} so that 
the map $L^s_1(V) \to L_1^s(F)$, which is also the map $L^{gs}_1(V) \to L_1^{gs}(F)$ \cite{CDH+III}, is the non-injective map $L_2^s(k) = W(k) \to 0$.
\end{remark}

To prove vanishing of $HL$ in degrees $\geq -2$ (respectively $\geq -3$), let us first record the following.

\begin{proposition}
\label{prop:2invertible}
Let $R$ be a ring with $2\in R^*$. Then the comparison map $L^{gs}(R) \to L^s(R)$ is an equivalence.
In particular, $HL(R)=0$.
\end{proposition}

\begin{proof}
This is for instance \cite[Example 3.2.8]{CDH+I}.
\end{proof}

By Proposition \ref{prop:ZarDesent}, it suffices to prove $H_n(R)=0$ for all regular local rings $(R,m,k)$ as in Theorem \ref{thm:MainL} and $n\geq -2$.
We will do so by induction on the dimension of the local ring, the induction step is provided by the following.

\begin{proposition}
\label{prop:MiniGerstenConj}
Let $(A,m_A)$ be a local ring and consider the local $A$-algebra $(B,m_B) = (A[X]_{(m_A,X)}, (m_A,X))$.
Then the map $HL_n(B) \to HL_n(B[X^{-1}])$ is injective for all integers $n\in \Z$.
\end{proposition}

\begin{proof}
Consider the projective line $p: \ppp^1_A \to \Spec A$, the sections $s_t:\Spec A \to \ppp^1_A$ for $t=0,\infty$ and the open inclusion $j:\aaa^1_A = \ppp^1_A-s_0 \subset \ppp^1_A$.
We have the commutative diagrams
$$\xymatrix{
HL_n(\ppp^1_A\on s_0) \ar[r] \ar[d]_{j^*}^{\cong} &  HL_n(\ppp^1_A)  \ar[d]_{j^*}  \ar[r]^{s_{\infty}^*} & HL_n(A) \ar[dl]^{p^*}  \\
HL_n(\aaa^1_A\on s_0)  \ar[r] & HL_n(\aaa^1_A) &}$$
where the vertical maps are restrictions and the horizontal maps are extension of support.
The left square commutes by functoriality, and 
the right triangle commutes as it commutes when precomposed with the isomorphism of Proposition \ref{prop:P1bdl}.
$$(p^*,\beta):HL_n(A) \oplus HL_{n+1}(A) \cong HL_n(\ppp^1_A)$$
 noting that $j^*\beta = s_t^*\beta=0$, by Lemma \ref{lem:GW10ZX}.
The left vertical map is an isomorphism, by Zariski excision.
Since the composition of the two top horizontal maps is zero, so is the lower horizontal map.
Now, consider the commutative square
$$\xymatrix{
 HL_n(A[X]\on X) \ar[r]^0 \ar[d]_{\cong} &  \ar[d] HL_n(A[X]) & \\
 HL_n(A[X]_{(m_A,X)}\on X)\ar[r] & HL_n(A[X]_{(m_A,X)}) \ar[r]&HL_n(A[X]_{(m_A,X)}[X^{-1}]) }$$
where the left vertical arrow is an isomorphism, by Zariski excision, and the top horizontal map is zero, as shown above. 
Therefore, the bottom left horizontal arrow is zero.
Since the bottom row is exact, the bottom right horizontal map is injective.
\end{proof}

In order to prove the local ring case of Theorem \ref{thm:MainL}, we will use the Cohen Structure Theorems for regular local rings.
Recall \cite[Theorem 29.7]{Matsumura} that the completion $\hat{R}$ of an equicharacteristic Noetherian regular local ring $R$ is a power series ring $\hat{R}\cong K[[X_1,...,X_d]]$ over a field $K$, called coefficient field of $\hat{R}$.
In the mixed characteristic case, we will use the following variant.
Recall \cite{FoxbyEtc} that a local homomorphism of local rings $(A,m_A) \to (B,m_B)$ is called {\em weakly regular} if it is flat and the closed fibre $B/m_AB$ is regular.
If $A$ is a dvr with uniformizer $\pi$, then $B$ is weakly regular over $A$ if and only if $B$ is regular, $A \subset B$, and $\pi \in m_B\setminus m_m^2$.

\begin{proposition}
\label{prop:CohenFlattening}
Let $(R,m,k_R)$ be a $d$-dimensional Noetherian regular local ring that is weakly regular over a dvr.
Then its completion $\hat{R}$ at the maximal ideal $m$ is a power series ring $\hat{R}\cong \ooo[[X_2,...,X_d]]$ over a complete dvr $\ooo$.
\end{proposition}

\begin{proof}
By assumption, $d\geq 1$.
The equicharacteristic case is a special case of the Cohen Structure Theorems cited above.
So, assume that $R$ is of mixed characteristic $(0,p)$ for a prime $p\in \Z$.
By assumption, $R$ contains a dvr $(V, \pi)$ of mixed characteristic $(0,p)$ such that $V \to R$ is weakly regular.
Let $C_R \subset \hat{R}$ and $C_{V}\subset \hat{V}$ be coefficient rings of their completions \cite[Theorem 29.7]{Matsumura}.
These are complete dvrs with uniformizing element $p$ such that the maps on residue fields are isomorphisms $C_{V}/p \cong \hat{V}/\pi$ and $C_R/p \cong \hat{R}/m_{\hat{R}}$.
Then $\hat{V}$ is an Eisenstein extension of $C_{V}$ \cite[Theorem 29.8]{Matsumura}, that is, there is a polynomial $f = T^n+a_{n-1}T^{n-1} + \cdots a_1T + a_0\in C_{V}[T]$ such that $p\mid a_i$ for $i=0,...,n-1$, $p^2 \nmid a_0$ and $\hat{V}\cong C_V[T]/f$.
Let $x\in \hat{V}$ be the element corresponding to $T\in C_V[T]/f$.
This is a uniformizer for the dvr $\hat{V}$.
In particular, as an element of $\hat{R}$, we have $x \in m_{\hat{R}}\setminus m^2_{\hat{R}}$.
The ring homomorphism $C_{V} \to C_{V}/p  = V/\pi \to R/m_R = C_R/p$ lifts to a local homomorphism $C_{V} \to C_R$ \cite[Theorem 29.2]{Matsumura}.
Now, we set $\ooo=C_R[T]/f$ and consider the local $C_V$-algebra homomorphism $\ooo \to \hat{R}$ that sends $T$ to $x\in \hat{V}\subset \hat{R}$.
As an Eisenstein extension of the complete dvr $C_R$, the ring $\ooo$ is a complete dvr with uniformizer $x\in \ooo$ and residue field $C_R/p\cong \ooo/x$.
The local homomorphism $\ooo \to \hat{R}$ is an injection into a domain, hence flat, and it is weakly regular because $x \in m_{\hat{R}}\setminus m^2_{\hat{R}}$.
Since $0 \neq x \in k_R=m_{\hat{R}}/m^2_{\hat{R}}$, we can extend $x\in \hat{R}$ to a regular sequence of parameters $x,x_2,x_3,...,x_d$ of $\hat{R}$ and define a local homomorphism of complete $\ooo$-algebras $\ooo[[X_2,...,X_d]] \to \hat{R}$ sending $X_i$ to $x_i$.
Since this is a homomorphism of complete regular local rings that is an isomorphism on residue fields and sends a regular system of parameters to a regular system of parameters, this map is an isomorphism.
\end{proof}

\begin{proposition}
\label{prop:LocalL}
Let $R$ be a regular local ring.
If $R$ contains a field of characteristic $2$ then $HL_n(R)=0$ for $n\geq -2$.
If $R$ contains a dvr $V$ of mixed characteristic $(0,2)$ such that $R$ is weakly regular over $V$, then
$HL_n(R)=0$ for $n\geq -3$.
\end{proposition}

\begin{proof}
We prove the statement by induction on the dimension $d$ of the regular local ring $R$.
The case $d\leq 1$ is Corollary \ref{cor:LvsLDVR}.
So, we can assume $d\geq 2$.
Let $\hat{R}$ be the completion of $R$ at its maximal ideal $m$ and write $X=\spec R$ and $\hat{X}=\spec \hat{R}$ with closed points $z\in X$ and $\hat{z} \in \hat{X}$.
By Flat Excision (Proposition \ref{prop:FlatExc}), we have an exact sequence
$$HL_{n+1}(\hat{X}-\hat{z})  \to HL_n(R) \to HL_n(X-z) \oplus HL_n(\hat{R}) \to HL_{n}(\hat{X}-\hat{z}).$$
Note that the local rings of $R$ and of $\hat{R}$ either contain a field or are weakly regular over $V$.
Thus, the induction hypothesis together with Zariski-descent (Proposition \ref{prop:ZarDesent}) imply that $HL_n(X-z) = HL_{n}(\hat{X}-\hat{z})=0$ for $n\geq -2$ (for $n\geq -3$ in the mixed characteristic case).
In particular, $HL_n(R) \cong HL_n(\hat{R})$ for $n \geq -2$ (for $n\geq -3$ in the mixed characteristic case).
By the Cohen Structure Theorems for equicharacteristic complete regular local rings and its mixed characteristic variant in Proposition \ref{prop:CohenFlattening}, we have $\hat{R}\cong \ooo[[X_2,...,X_d]]$ for a complete dvr $(\ooo,\pi)$ of residue characteristic $2$.
By the argument above applied to the completion $\ooo[[X_2,...,X_d]]$ of the regular local ring $\ooo[X_2,...,X_d]_M$ for $M=(\pi,X_2,...,X_d)$, 
we have $HL_n(\ooo[[X_2,...,X_d]]) = HL_n(\ooo[X_2,...,X_d]_M)$ for $n\geq -2$ (for $n\geq -3$ in the mixed characteristic case).
Consider the ring $A=\ooo[X_2,...,X_{d-1}]$ and the maximal ideal $m_A=(\pi,X_2,...,X_{d-1})$.
Then $(\ooo[X_2,...,X_d]_M,M) = (A[X]_{(m_A,X)},(m_A,X))$ with $X=X_d$.
By Proposition \ref{prop:MiniGerstenConj}, the map 
$$HL_n(\ooo[X_2,...,X_d]_M) \hookrightarrow HL_n(\ooo[X_2,...,X_d]_M[X_d^{-1}])$$ is injective for all integers $n\in \Z$.
We note that the $d-1$-dimensional regular ring $\ooo[X_2,...,X_d]_M[X_d^{-1}]$ has local rings that contain a field or are weakly regular over $\ooo$.
Hence, it has vanishing $HL_n$ for $n\geq -2$ (for $n\geq -3$ in the mixed characteristic case), by induction hypothesis and Zariski descent (Proposition \ref{prop:ZarDesent}).
Hence, for $n\geq -2$ (for $n\geq -3$ in the mixed characteristic case), we have
$$HL_n(R) \cong HL_n(\hat{R})= HL_n(\ooo[[X_2,...,X_d]]) \subset  HL_n(\ooo[X_2,...,X_d]_M[X_d^{-1}])=0.$$
\end{proof}

\begin{theorem}
\label{thm:MainLText}
Let $R$ be a regular Noetherian domain of finite Krull dimension.
\begin{enumerate}
\item
If $0=2\in R$, then the comparison map
$$L_n^{gs}(R) \to L_n^s(R)$$
is an isomorphism for $n \geq -1$ and a monomorphism for $n=-2$.
\item
If $0\neq 2\in R$, assume that the $(0,2)$ mixed characteristic local rings of $R$ are weakly regular over a dvr. 
Then the comparison map
$$L_n^{gs}(R) \to L_n^s(R)$$
is an isomorphism for $n \geq -2$ and a monomorphism for $n=-3$.
\end{enumerate}
\end{theorem}

\begin{proof}
Note that the local rings of $R$ either contain a field, have $2$ as a unit, or are weakly regular over a dvr of mixed characteristic $(0,2)$.
Now, the result follows from Proposition \ref{prop:LocalL} in view of Proposition \ref{prop:ZarDesent}.
\end{proof}

\begin{corollary}
\label{cor:MainGWText}
Let $X$ be a finite dimensional regular Noetherian integral separated scheme.
\begin{enumerate}
\item
If $X$ is defined over $\F_2$, then the comparison maps from genuine symmetric to symmetric Hermitian $K$-theory
$$GW^{[r]}_n(X) \to KO^{[r]}_n(X)$$
are isomorphisms for $n \geq r-1$ and a monomorphism for $n=r-2$.
\item
If $0 \neq 2 \in \Gamma(X,O_X)$, 
assume that the $(0,2)$ mixed characteristic stalks $O_{X,x}$ are weakly regular over a dvr.
Then the comparison maps from genuine symmetric to symmetric Hermitian $K$-theory
$$GW^{[r]}_n(X) \to KO^{[r]}_n(X)$$
are isomorphisms for $n \geq r-2$ and a monomorphism for $n=r-3$.
\end{enumerate}
\end{corollary}

\begin{proof}
By Zariski descent, we can assume that $X=\Spec(R)$ is affine.
Now, the theorem follows from Theorem \ref{thm:MainLText} since the fibre of $GW^{[r]}\to KO^{[r]}$ is the fibre of the map of spectra $S^r\wedge L^{gs} \to S^r\wedge L^s$, by \cite[Main Theorem]{CDH+II} and the identifications $L(R,\Qoppa^{[r]}) = S^r\wedge L(R,\Qoppa)$.
\end{proof}

\begin{theorem}
\label{thm:HtpyInvText}
Let $R$ be a regular Noetherian ring of finite Krull dimension.
Assume that the mixed characteristic $(0,2)$ local rings of $R$ are weakly regular over a dvr. 
Then the inclusion $R \subset R[T]$ induces homotopy equivalences of genuine symmetric and symplectic Hermitian K-theory spaces
$$\GW(R) \stackrel{\sim}{\longrightarrow} \GW(R[T]),\hspace{6ex} \KSp(R) \stackrel{\sim}{\longrightarrow} \KSp(R[T]).$$
\end{theorem}

\begin{proof}
The statements for $\pi_n\GW =GW^{[0]}_n$, $n\geq 0$, and for $\pi_n\KSp=GW^{[2]}_n$, $n\geq 1$, follow from Corollary \ref{cor:MainGWText} and the fact that $KO$ is homotopy invariant on regular rings \cite[Theorem 6.3.1]{CHN}.
The case $K_0Sp$ is known but also follows from the commutative diagram
$$\xymatrix{GW^{[2]}_0(R[X]) \ar@{->>}[d] \hspace{1ex} \ar@{>->}[r]  & KO^{[2]}_0(R[X])\ar@{=}[d]\\
GW^{[2]}_0(R) \hspace{1ex} \ar@{>->}[r]  & KO^{[2]}_0(R)}
$$ 
with vertical maps induced by the split surjective $R$-algebra homomorphism $R[X]\to R:X\mapsto 0$.
It follows that the left surjective map is also injective.
\end{proof}

\begin{remark}
\label{rmk:motivicRep}
Let $S$ be a regular ring.
The motivic Hermitian $K$-theory spectrum $KQ _S\in \mathcal{SH}(S)$ of \cite{CHN} is the sequence $(KO^{[0]}, KO^{[1]}, KO^{[2]}, KO^{[3]}, ...)$ of presheaves of $S^1$-spectra on smooth $S$-schemes with bonding maps given by the projective line bundle formula.
It is equivalently given by the sequence 
$$(\Omega^{\infty}_{S^1}\, KO^{[0]}, \Omega^{\infty}_{S^1}\, KO^{[1]}, \Omega^{\infty}_{S^1}\, KO^{[2]}, \Omega^{\infty}_{S^1}\, KO^{[3]}, ...)$$ 
of motivic spaces over $S$ with bonding maps induced by the projective line bundle formula.
Note that this sequence is $4$-periodic since $KO^{[n]}\simeq KO^{[n+4]}$.
It is $2$-periodic in characteristic $2$.

Denote by $\GW^{[n]}=\Omega^{\infty}_{S^1}\, GW^{[n]}$ the $n$-th shifted genuine symmetric Hermitian $K$-theory space.
So, $GW^{[0]}(R)=\GW(R)$ and $\GW^{[2]}(R)=\KSp(R)$ are the genuine symmetric and symplectic Hermitian $K$-theory spaces of $R$ from the introduction.
In view of Corollary \ref{cor:MainGWText}, if $\F_2 \subset S$, then the map from the  $\ppp^1$-spectrum given by the $2$-periodic sequence of genuine Hermitian $K$-theory spaces
$$(\GW^{[0]}, \GW^{[1]}, \GW^{[0]}, \GW^{[1]}, ... )$$
to $KQ_S$ is a motivic equivalence as it is level-wise a simplicial equivalence.
If $S$ is a Dedekind domain with $0\neq 2\in S$, then 
the map from the $\ppp^1$-spectrum given by the $4$-periodic sequence of genuine Hermitian $K$-theory spaces
\begin{equation}
\label{eq:rmk:motivicRep}
(\GW^{[0]}, \GW^{[1]}, \GW^{[2]}, \GW^{[3]}, \GW^{[0]}, \GW^{[1]}, \GW^{[2]}, \GW^{[3]}, ... )
\end{equation}
to $KQ_S$ is a motivic equivalence as it is level-wise a simplicial equivalence except, possibly, at the levels $3\mod 4$ where it is an equivalence in degrees $\geq 1$ and a monomorphism in degree $0$.
\end{remark}

\bibliographystyle{plain}

\newcommand{\etalchar}[1]{$^{#1}$}

\end{document}